\documentclass[11pt]{article} 
\usepackage{amssymb, latexsym, amsmath,xcolor, amsthm}
\usepackage{authblk}
\newtheorem{theorem}{Theorem}[section]
\newtheorem{lemma}[theorem]{Lemma}

\newtheorem{remark}[theorem]{Remark}
\theoremstyle{definition}
\newtheorem{example}{Example}[section]

\newcommand{\ra}{\rightarrow}

\newcommand{\ms}{\mapsto}
\newcommand{\ol}{\overline}

\newcommand{\N}{\mathbb N}

\newcommand{\Z}{\mathbb Z}
\newcommand{\Q}{\mathbb Q}
\renewcommand{\P}{\mathcal P}

\newenvironment{items}{\begin{list}{$\alph{item})$}
{\labelwidth18pt \leftmargin18pt \topsep3pt \itemsep1pt \parsep0pt}}
{\end{list}}

\author{
Bettina Eick\thanks{Institut für Analysis und Algebra,
Technische Universität Braunschweig,
38106 Braunschweig, Germany.
E-mail: \texttt{b.eick@tu-braunschweig.de}} \; and 
Renu Joshi \thanks{E-mail: \texttt{renu.joshi@tu-braunschweig.de}}
}

\begin{document}
\title{Derivations of Lie $p$-rings with applications \\
   on ring extensions and computational methods}
   \date{\today}

\maketitle
\begin{abstract}
Derivations are a well-known concept in the theory of Lie algebras. Here we 
investigate them for Lie $p$-rings. We show how they apply in the 
construction of cyclic extensions and in computational methods; in 
particular, we demonstrate that a consistency test for Lie $p$-presentations 
can be facilitated using derivations. The consistency test generalizes in
specific cases to parametrised Lie $p$-presentations and, further, to the
construction of Lie $p$-covers. These provide a first building block towards 
a symbolic Lie $p$-ring generation algorithm. 
\end{abstract}



\section{Introduction}

Given a Lie algebra $L$, the set of its derivations $\mathrm{Der}(L)$ forms 
a Lie algebra containing the inner derivations $\mathrm{Inn}(L)$ as an 
ideal. Derivations of a Lie algebra have applications in the construction of 
extensions and general investigations of Lie algebras. Their role is 
similar to that of automorphism groups in group theory. We refer to the 
book by de Graaf \cite{Gr00} for details.

A Lie $p$-ring is a nilpotent Lie ring of $p$-power order. The construction 
of the derivations $\mathrm{Der}(L)$ extends naturally to a Lie $p$-rings. 
In Section \ref{deriv}, we define and investigate these derivations, and 
show how they can be used to construct {\em cyclic extensions} of a Lie 
$p$-ring $L$; that is, Lie rings $H$ containing $L$ as ideal with quotient 
$H/L$ of size $p$. 

The Lazard correspondence translates the Lie $p$-rings of size $p^n$
to the groups of order $p^n$ and vice versa; if $p > n$, then this yields 
a one-to-one correspondence between their isomorphism types. This 
correspondence adds to the interest in constructing the Lie $p$-rings 
of order $p^n$, where $n$ is fixed, but $p$ remains indeterminate. We 
refer to Higman \cite{Hig60} for further background. 

Given a prime $p$ and $n \in \N$, the groups of order $p^n$ can in principle
be constructed using the $p$-group generation algorithm by O'Brien 
\cite{OBr90}. This algorithm considers all groups $G$ of order $p^m$ with 
$m < n$ in turn. Given one such $G$, it computes all {\em descendants} of 
$G$; that is, the groups $H$ so that $H / \lambda(H) \cong G$, where
$\lambda(H)$ is the smallest non-trivial term of the lower $p$-central 
series of $H$. The computation of descendants of a group $G$ is based
on the computation of its $p$-cover $G^*$.

Included below is a method to compute the Lie $p$-cover $L^*$ of a Lie $p$-ring
$L$. The main ideas of this method for Lie $p$-rings are similar of that for 
$p$-groups. We include a new elementary proof for the correctness of the
method for Lie $p$-rings; this is based on the construction
of cyclic extensions for Lie $p$-rings. Moreover, we show how the 
$p$-cover algorithm for Lie $p$-rings generalizes to parametrised (or
symbolic) Lie $p$-rings; that is, Lie $p$-rings with $p$ indeterminate 
and possible further parametrisations. This variation does not determine 
a single Lie $p$-cover, but a finite family of such, describing the 
possible Lie $p$-covers symbolically.

The Lie $p$-cover algorithm developed here applies to the symbolic Lie 
$p$-rings of order dividing $p^7$ determined by O'Brien and Vaughan-Lee 
\cite{OBVL05}, which are available electronically in the Lie $p$-ring 
package \cite{VLE15} of the computer algebra system GAP \cite{GAP}. An
implementation of the algorithm developed here will be made available in
this package. It provides a first step towards a symbolic algorithm to 
determine the Lie $p$-rings of size $p^8$.

\section{Lie $p$-rings and presentations}
\label{liepres}

We recall the basic theory of Lie $p$-rings and their Lie $p$-presentations.
A \emph{Lie ring} is an additive abelian group with a bracket operation
$[\cdot,\cdot]$ which is bilinear, alternating, and satisfies
the Jacobi-identity. Given a Lie ring $L$, its lower central series
is defined by $L = L_1 > L_2 > \cdots$ with $L_{i+1} = [L, L_i]$. The
Lie ring $L$ is nilpotent, if its lower central series terminates at
$\{0\}$; the length of this series is then the {\em class} of $L$.

A {\em Lie $p$-ring} is a nilpotent Lie ring with $p^n$ elements for
some $n \in \N$ and a prime $p$.  Given a Lie $p$-ring $L$, its lower 
$p$-central series
is defined by $L = \Lambda_1(L) > \Lambda_2(L) > \cdots$ with $\Lambda_{i+1}(L) = [L, \Lambda_i(L)] + p\Lambda_i(L)$.
The nilpotency of $L$ implies that the lower $p$-central series terminates
at $\{0\}$; the length of this series is the {\em $p$-class} of $L$.

If $L$ is a Lie $p$-ring of order $p^n$, then we often denote $n$ as the
{\em dimension} of $L$ (despite the fact that $L$ is not a vector space).
We also say that $l_1, \ldots, l_n$ is a {\em basis} of $L$ if the series
$L = U_1 > U_2 > \cdots$ defined by $U_i = \langle l_i, \ldots, l_n \rangle$
for $1 \leq i \leq n+1$ is strictly decreasing and refines the lower 
$p$-central series.

\subsection{Free Lie rings}

We briefly recall an explicit construction for the free Lie ring $F_n$
on $n$ generators $l_1, \ldots, l_n$ over $R$, where $R$ is a commutative
ring with $1$. Let \(X=\{l_1,\dots,l_n\}\), and let \(M(X)\) be the free 
magma on \(X\), i.e. the set of all  non-associative bracketed expressions 
in the letters \(l_i\). By definition, $X \subseteq M(X)$. The free magma 
has a natural binary operation
\[ m: M(X) \times M(X) \to M(X), \qquad (a,b) \mapsto a \star b.  \]
Let $A_R(X)$ be free non-associative  algebra over $R$ generated by 
\(X\), i.e., $A_R(X)$ is the $R$-span of $M(X)$.
The multiplication on $A$ is obtained by $R$-bilinearly extending the map 
$m$. Let \(J\) be the ideal of \(A_R(X)\) generated by all elements
\[
a\star a,\quad
a\star b+ b\star a, \quad\text{and} \quad
(a\star(b\star c))+(c\star(a\star b))+(b\star(c\star a)),
\]
for all \(a,b,c\in M(X)\). Then the quotient $F_n \cong A_R(X)/J$ is the 
free Lie ring on generators \(l_1,\dots,l_n\) over $R$. Note that the 
multiplication induced by \(\star\) on the quotient is bilinear, 
alternating, and satisfies the Jacobi identity, hence \(F_n\) is indeed 
a Lie ring.

\subsection{Lie $p$-presentations}

Let $p$ be a prime and $R = \Z$.
Write $\Z_p = \{ 0, \ldots, p-1\}$ and $A = (a_{i,j,k} \mid 1 \leq i \leq
j < k \leq n)$ be an array with entries $a_{i,j,k} \in \Z_p$. Let $\mathcal{P}(A)$
denote the presentation on $l_1, \ldots, l_n$ subject to the relations
\begin{eqnarray*}
  pl_i      &=& \sum_{k=i+1}^n a_{i,i,k} l_k
                \;\; \mbox{ for } \;\; 1 \leq i \leq n, \\
 \left [ l_i, l_j \right ] &=& \sum_{k=j+1}^n a_{i,j,k} l_k
                \;\; \mbox{ for } \;\; 1 \leq i < j \leq n. \\
\end{eqnarray*}

Then $\mathcal{P}(A)$ is called a {\em Lie $p$-presentation}. This 
notation is justified by the following.

\begin{lemma}\label{lem1}
Let $F_n$ be the free Lie ring on $n$ generators over $\Z$, and let $I(A)$ be the ideal in $F_n$ defined by the relations in $\mathcal{P}(A)$. Then $L(A) = F_n/I(A)$ is a Lie $p$-ring with 
at most $p^n$ elements.
\end{lemma}

\begin{proof}
By construction, $L(A)$ is a Lie ring. The relations with right-hand-side
$p l_i$ ensure that $L(A)$ has at
most $p^n$ elements and these can be
described by $x_1 l_1 + \cdots + x_n l_n$ with $x_i \in \Z_p$. The relations
with right-hand-side $[l_i, l_j]$ yield that $L(A)$ is nilpotent, since
$[l_i, l_j] \in \langle l_{j+1}, \ldots, l_n \rangle$ for each $i<j$.
In summary, $L(A)$ is a Lie $p$-ring.
\end{proof}

A Lie $p$-presentation $\mathcal{P}(A)$ is called {\em consistent} if its associated
Lie $p$-ring $L(A)$ has precisely $p^n$ elements. We observe the following.

\begin{lemma}
Each Lie $p$-ring $L$ of dimension $n$ can be defined by a consistent
Lie $p$-presentation $\mathcal{P}(A)$ on $n$ generators.
\end{lemma}

\begin{proof}
Let $l_1, \ldots, l_n$ be a basis of $L$. Then each element of $L$ can
be uniquely described in the form $x_1 l_1 + \cdots + x_n l_n$ with
$x_i \in \Z_p$; we call this the {\em normal form} of an element in
$L$. Further, $p l_i \in \langle l_{i+1}, \ldots, l_n \rangle$ and
$[l_i, l_j] \in \langle l_{j+1}, \ldots, l_n \rangle$ for each $i < j$,
since a basis exhibits the lower $p$-central series of $L$. Expressing each element $pl_i$ and
$[l_i,l_j]$ in its normal form determines coefficients
$a_{i,i,k}$ and $a_{i,j,k}$, and hence defines a Lie $p$-presentation
$\mathcal P(A)$. By construction, the generators $l_1,\ldots,l_n$ satisfy the defining
relations of $\mathcal P(A)$ in the Lie $p$-ring $L$. Therefore, the
natural homomorphism from the Lie ring defined by $\mathcal P(A)$ onto
$L$ is surjective. By Lemma \ref{lem1}, the presented Lie $p$-ring has at most $p^n$ elements, while $L$ has exactly
$p^n$ elements. Hence the homomorphism is an isomorphism, and
$\mathcal P(A)$ is a consistent Lie $p$-presentation of $L$.
\end{proof}

\subsection{Symbolic Lie $p$-presentations}
\label{symb}

Let $p$ be an indeterminate, $\ell \in \N$ and $w, x_1, \ldots, x_m$ 
commuting indeterminates. Define $R_\ell = \Z[1/\ell][w, x_1, \ldots, x_m]$;
then $R_\ell$ is an integral domain. Let $Z$ be an ideal in $R_\ell$ and 
$U$ a multiplicatively closed subset of $R_\ell \setminus \{0\}$. We write
\[ R_\ell(U) = \{ \frac{f}{g} \mid f \in R_\ell, g \in U\}.\]

Let $A=(a_{i,j,k}\mid 1\le i\le j<k\le n)$ be an array with entries 
$a_{i,j,k}\in R_\ell(U)$. As before, let $\mathcal{P}(A)$ denote the 
presentation on $l_1, \ldots, l_n$ subject to the relations
\begin{eqnarray*}
  p l_i      &=& \sum_{k=i+1}^n a_{i,i,k} l_k
                \;\; \mbox{ for } \;\; 1 \leq i \leq n, \\
 \left [ l_i, l_j \right ] &=& \sum_{k=j+1}^n a_{i,j,k} l_k
                \;\; \mbox{ for } \;\; 1 \leq i < j \leq n. \\
\end{eqnarray*}
Then $\mathcal{P}(A)$ is a {\em symbolic Lie $p$-presentation}. A symbolic 
Lie $p$-presentation defines a family of  Lie $p$-presentations via
evaluation. For this, let $P$ be a prime with $\gcd(P,\ell) = 1$ and let
$W$ be the smallest primitive root mod $P$. Define $\Sigma_P(U,Z)$  to be 
the set of all tuples $X = (X_1, \ldots, X_m) \in \Z_P^m$ so that
\begin{eqnarray*}
&& h(W,X_1, \ldots, X_m) \not\equiv 0 \pmod {P} \mbox{ for all } h \in U, 
     \mbox{ and } \\
&& z(W,X_1, \ldots, X_m) \equiv 0 \pmod {P}  \mbox{ for all } z \in Z.
\end{eqnarray*}
Each entry in the array $A$ has the form $f/g$ with $f \in R$ and $g \in
U$. We evaluate such an entry at $X \in \Sigma_P(U,Z)$ via 
$f(W,X_1, \ldots, X_m)/g(W, X_1, \ldots, X_m) \bmod P$. We evaluate
the array $A$ by evaluating each entry. Then $\mathcal{P}(A)$ evaluated at a prime $P$ and $X \in \Sigma_P(U,Z)$ is 
the Lie $p$-presentation $\mathcal{P}(\ol{A})$, where $\ol{A} = A(W,X)$ is 
the evaluation of $A$ at $X$. By construction and Lemma \ref{lem1}, 
each such evaluation defines a Lie $p$-ring with at most $P^n$ elements.

We say that the symbolic Lie $p$-presentation $\mathcal{P}(A)$ is
{\em consistent} with respect to a set of primes $\Pi$ if its evaluation 
is consistent for each choice of a prime $P \in \Pi$ with $\gcd(P, \ell) 
= 1$ and each choice of $X \in \Sigma_P(U,Z)$.
                
\section{Derivations, extensions and consistency}
\label{deriv}

We now introduce derivations for Lie $p$-rings and cyclic extensions 
together with their application to consistency checking.

\subsection{Derivations}

Let $L$ be a Lie $p$-ring. We say that a map $D : L \ra L$ is {\em linear}
if it is a group homomorphisms with respect to addition. The map $D$ is a
{\em derivation} if it is linear and satisfies
\[ D([x,y]) = [D(x),y] + [x,D(y)] \;\; \mbox{ for all } x,y \in L.\]

\begin{remark}
The set $\mathrm{Der}(L)$ of all derivations of $L$ is a finite Lie ring with
bracket multiplication $[D_1, D_2] = D_1 D_2 - D_2 D_1$. 
\end{remark}

The {\em adjoint} of $x \in L$ is the map $\mathrm{ad}_x : L \ra L 
: l \ms [x,l]$. We note that $\mathrm{ad}_x$ is a derivation for each 
$x \in L$ (called inner derivation induced by $x$). The set 
$\mathrm{Inn}(L) = \{ \mathrm{ad}_x \mid x \in L \}$ is an ideal in the
Lie ring $\mathrm{Der}(L)$ and $\mathrm{Inn}(L) \cong L/Z(L)$, where $Z(L) = \{ x \in L 
\mid [x,l] = 0 \mbox{ for all } l \in L\}$. 

A {\em cyclic extension} of a Lie $p$-ring $L$ is a Lie ring $H$
satisfying a short exact sequence of the form
\[ \{0\} \ra L \ra H \ra C \ra \{0\},\]
where $C$ is the (unique) 1-dimensional Lie $p$-ring; that is, $C =
\langle c \mid pc = 0 \rangle$. The following lemma exhibits a link
betweeen cyclic extensions and derivations.

\begin{lemma}
\label{cycext1}
Let $H$ be a cyclic extension of a Lie $p$-ring $L$ and let 
$\sigma : H \ra C$ be the associated Lie ring epimorphism with 
kernel $L$. Let $h \in H$ so that $C = \langle \sigma(h) \rangle$. 
Then
\begin{items}
\item[\rm (a)]
$D : L \ra L : y \ms [h,y]$ is a derivation of $L$.
\item[\rm (b)]
$p D : L \ra L : y \ms p [h,y]$ is the inner derivation 
$\mathrm{ad}_{ph}$ of $L$.
\end{items}
\end{lemma}

\begin{proof}
(a) Since \(L=\ker(\sigma)\) is an ideal of \(H\), we have $[h,y]\in L$
for every \(y\in L\). Hence \(D\) is well defined. Further, the Lie bracket 
is bilinear, and thus
\[ D(y+z) =[h,y+z] =[h,y]+[h,z] =D(y)+D(z) \]
for all \(y,z\in L\). Hence \(D\) is linear. Now, let \(y,z\in L\). Using 
the Jacobi identity and antisymmetry, we have that
\begin{eqnarray*}
 D([y,z]) &=& [h,[y,z]]=-[z,[h,y]]-[y,[z,h]] =[[h,y],z]+[y,[h,z]]  \\
          &=& =[D(y),z]+[y,D(z)].
\end{eqnarray*}
Hence \(D\) is a derivation of \(L\) and (a) is proved. \\
(b) Let \(y\in L\). Then $(pD)(y)=pD(y)=p[h,y]$.  Using the bilinearity 
of the Lie bracket, we obtain $p[h,y]=[ph,y]$.  Hence
\[ (pD)(y)=[ph,y]=\mathrm{ad}_{ph}(y).  \]
Since this holds for every \(y\in L\), we conclude that
$pD=\mathrm{ad}_{ph}$.
\end{proof}

The derivation $D$ in Lemma \ref{cycext1} is not unique, as it depends
on the choice of $h$. A different choice has the form $h' = kh +y$ for
some $k \in \{0, \ldots, p-1\}$ and $y \in L$, and thus yields an alternative 
derivation $D' = k D + \mathrm{ad}_y$. Note that $D' \equiv kD \bmod
\mathrm{Inn}(L)$.

\subsection{Cyclic extensions}

The following theorem provides a converse to Lemma~\ref{cycext1} by 
constructing a cyclic extension from an appropriate derivation.

\begin{theorem}
\label{thm4}
Let $L$ be a Lie $p$-ring, and let $D \in \mathrm{Der}(L)$ so that 
$pD = \mathrm{ad}_w \in \mathrm{Inn}(L)$ for some $w \in L$ and $D(w) = 0$.
Define
\[ H = \{ (x, l) \mid 0 \leq x < p, l \in L \} \]
with the operations
\begin{items}
\item[\rm (i)]
$(x, l) + (y, k) = (r, qw+l+k)$ if $x+y = qp+r$ with $0 \leq r < p$, and
\item[\rm (ii)]
$[(x, l),(y, k)] = (0, xD(k) - yD(l) + [l,k])$.
\end{items}
Then $H$ is a cyclic extension of $L$. 
\end{theorem}

\begin{proof}
We first note that the addition defined in (i) turns $H$ into an additive
abelian group (independent of $D$); see also Kapitel 1, Beispiel 14.8 in
Huppert \cite{Hup67}. Next, we verify that the operation (ii) defines a 
Lie bracket on \(H\). Let $a=(x,l)$, $b=(y,k)$, and $d=(z,m)$ be elements 
of $H$. Note that the Lie bracket is alternating, since $L$ is a Lie ring.
We consider the bilinearity and Jacobi identity in the following.
\medskip

\noindent
{\bf Bilinearity:} We consider linearity in the first component. Write 
$x+y = qp+r$ with $r \in \Z_p$. Then (keeping in mind that $D(w)=0$
and noting that $rD(m) + q[w,m] = (x+y) D(m)$)
\[
\begin{aligned}
[a+b,d]
&= [(r, qw+l+k),(z,m)] \\
&= (0,rD(m)-zD(q w+l+k)+[q w+l+k,m])\\
&= (0,rD(m)-zqD(w)-zD(l)-zD(k)+q[w,m]+[l,m]+[k,m]) \\
&= (0,rD(m)-zD(l)-zD(k)+q[w,m]+[l,m]+[k,m]) \\
&= (0,xD(m)+yD(m)-zD(l)-zD(k) +[l,m]+[k,m]), \mbox{ and } \\
[a,d]+[b,d]
&= (0,xD(m)-zD(l)+[l,m])  + (0,yD(m)-zD(k)+[k,m]) \\
&= (0,xD(m)+yD(m)-zD(l)-zD(k)+[l,m]+[k,m]).
\end{aligned}
\]
Thus $[a+b,d]=[a,d]+[b,d]$ follows. Linearity in the second component
follows with similar arguments.
\medskip

\noindent
{\bf Jacobi identity:} We evaluate the equation $[a,[b,d]] + [b,[d,a]]
+ [d,[a,b]]$. For this, first note that $[b,d] = (0,yD(m)-zD(k)+[k,m])$.
Hence
\[
\begin{aligned}
[a,[b,d]]
&= [(x,l),(0,yD(m)-zD(k)+[k,m])] \\
&= (0, xD(yD(m)-zD(k)+[k,m]) +[l,yD(m)-zD(k)+[k,m]]) \\
&= (0, xyD^2(m)-xzD^2(k) +x[D(k),m]+x[k,D(m)] \\
           &\qquad +y[l,D(m)]-z[l,D(k)] +[l,[k,m]]).
\end{aligned}
\]

Similarly, 
\[
\begin{aligned}
[b,[d,a]] &= (0, yzD^2(l)-xyD^2(m) +y[D(m),l]+y[m,D(l)] \\
         &\qquad +z[k,D(l)]-x[k,D(m)] +[k,[m,l]]), \mbox{ and } \\
[d,[a,b]] &= (0, xzD^2(k)-zyD^2(l) +z[D(l),k]+z[l,D(k)] \\
         &\qquad +x[m,D(k)]-y[m,D(l)] +[m,[l,k]]).
\end{aligned}
\]
As $L$ is a Lie ring, it follows that $[l,[k,m]] + [k,[m,l]] + [m,[l,k]]=0$.
Using antisymmetry, it now follows that $[a,[b,d]]+[b,[d,a]]+[d,[a,b]]=0$. 
Therefore, $H$ is a Lie ring. By construction, $H$ is a cyclic extension.
\end{proof}

\subsection{Consistency}

Let $\mathcal{P}(A)$ denote a Lie $p$-presentation with generators 
$l_1, \ldots, l_n$. The power relations in $\mathcal{P}(A)$  can be used to reduce an element $a_1 l_1 + \cdots + a_n l_n$ with $a_i \in \Z$
by iteratively applying the relations $p l_i$ as long as there is a 
coefficient $a_i \not \in \Z_p$. We apply the power relations {\em from 
the left to right}; that is, we begin with the term  $a_1 l_1$. Writing
$a_1 = qp+r$ with $r \in \Z_p$, we replace the given word by subtracting $q$ times the power 
relation $p l_1 - \sum a_{1,1,k} l_k$. Repeating this procedure, we eventually obtain a (unique) equivalent word 
$b_1 l_1 + \cdots + b_n l_n$ with $b_i \in \Z_p$. We write 
\[ RF(a_1 l_1 + \cdots + a_n l_n) = b_1 l_1 + \cdots + b_n l_n,\]
and we call the expression on the right-hand side the {\em power reduced form} of the
original expression.

\begin{remark}
Let $L$ be the Lie $p$-ring defined by Lie $p$-presentation $\P(A)$. 
If $\P(A)$ is consistent, then each element of $L$ has a unique power 
reduced form; that is, the power reduced form is a normal form in this 
case. If $\P(A)$ is not consistent, then different power reduced forms 
define the same element. 
\end{remark}

Using the relations of $\mathcal{P}(A)$, it follows that 
(denoting $a_{i,j,k} = 0$ if $k < j$)
\[ [l_i, [l_j, l_k]] 
  = \sum_h  a_{j,k,h} [l_i, l_h]
  = \sum_u (\sum_h a_{j,k,h} a_{i,h,u}) l_u.\]
We define 
\begin{eqnarray*}
 J(i,j,k) 
 &:=& 
 \sum_u (\sum_h a_{j,k,h} a_{i,h,u} 
               + a_{k,i,h} a_{j,h,u} 
               + a_{i,j,h} a_{k,h,u}) l_u  \\
 &=& [l_i, [l_j, l_k]] + [l_j, [l_k, l_i]] + [l_k, [l_i, l_j]], 
\end{eqnarray*}
and 
\begin{eqnarray*}
 P(i,j) 
  &:=& \sum_k (p a_{i,j,k} - \sum_h a_{h,j,k} a_{i,i,h}) l_k  \hspace{2cm} \\
  &=& \sum_k p a_{i,j,k} l_k - [\sum_h a_{i,i,h} l_h, l_j] \\
  &=& p[l_i, l_j] - [pl_i, l_j] \\
\end{eqnarray*} 
An important application of the cyclic extension construction is the following theorem, which forms the basis of an algorithm for testing the consistency of Lie $p$-presentations.

\begin{theorem}\label{thm5}
Let $\mathcal{P}(A)$ be a Lie $p$-presentation. Then 
$\mathcal{P}(A)$ is consistent if and only if
\begin{items}
\item[\rm (a)]
$RF(J(i,j,k)) = 0$ for all $1 \leq i, j, k \leq n$.
\item[\rm (b)]
$RF(P(i,j)) = 0$ for all $1 \leq i,j \leq n$. 
\end{items}
\end{theorem}

\begin{proof}
If $\mathcal{P}(A)$ is consistent, then it defines a Lie $p$-ring and 
thus the Jacobi identity and the compatibility of the Lie bracket with 
$p$-multiplication are satisfied. Hence (a) and (b) follow.

Conversely, assume that conditions (a) and (b) hold. We prove that 
$\mathcal{P}(A)$ is consistent by induction on 
the number of generators $n$. Consider the series 
$L_i = \langle l_i, \ldots, l_n \rangle$ for $1\leq i\leq n$. Note 
that $L_i$ is a Lie $p$-ring defined by $\mathcal{P}(A_i)$ for a 
subarray $A_i$ of $A$. The presentation $\mathcal{P}(A_n)$ is trivially 
consistent, so assume inductively that $\mathcal{P}(A_2)$
is consistent. Then $L_2$ has $p^{n-1}$ elements. The relations in
$\P(A)$ ensure that $L_2$ is an ideal in $L=L_1$. We define a map 
$D : L_2 \ra L_2$ via 
\[ D( a_2 l_2 + \cdots + a_n l_n ) = RF(\sum_{j=2}^n a_j [l_1, l_j]),\]
where $[l_1, l_j] = \sum a_{1jk} l_k$ as described by $\P(A)$. Since
$L_2$ is consistent, $D$ is a well-defined linear map.
We show that $D$ is a derivation. By condition (a) for 
$i=1$, we have $RF(J(1,j,k)) = 0$. As $J(1,j,k) \in L_2$, consistency of $L_2$ implies
that $J(1,j,k) = 0$. Hence
\begin{eqnarray*}
D([l_j, l_k]) 
  &=& [l_1, [l_j, l_k]] \\
  &=& -[[l_j, [l_k, l_1]] - [l_k, [l_1, l_j]] \\
  &=& [D(l_j), l_k] + [l_j, D(l_k)]. 
\end{eqnarray*}
Further, $pD = ad_{pl_1}$ by condition (b) for $i = 1$ and $j > 1$, and
$D(pl_1) = 0$ by condition (b) for $i=j=1$. Hence, in summary, all
conditions of Theorem \ref{thm4} are satisfied. Consequently, $L$
is a cyclic extension of $L_2$. It follows that $|L| = p |L_2| = p^n$ and 
hence $\P(A)$ is consistent.
\end{proof}

We now generalize Theorem \ref{thm5} to symbolic Lie $p$-presentations. 
In the symbolic case, $p$ is indeterminate, and the coefficients 
$a_{ijk}\in R_\ell(U)$. The main difference between the ordinary and the 
symbolic case is the computation of power reduced forms $RF$. Given 
$a_1 l_1 + \cdots + a_n l_n$ with $a_i \in R_\ell(U)[p]$, the power 
reduced form is obtained by repeatedly applying the following reductions.
\begin{items}
\item[$\bullet$]
$a_i$ is a polynomial in $p$ over the ring $R_\ell(U)$. Write $a_i = qp
+ a_i'$ with $a_i' \in R_\ell(U)[p]$ of smaller degree in $p$. Apply
the power relation with left hand side $p l_i$ to reduce.
\item[$\bullet$]
Now $a_i \in R_\ell(U)$; that is, $a_i = b_i/c_i$ with $b_i \in R_{\ell}$ 
and $c_i \in U$. Reduce $b_i$ modulo $Z$.
\end{items}
The first step is equivalent to the reduction procedure for a fixed prime. 
As before, the reductions are applied from left to right, beginning with 
the coefficient of $l_1$ and proceeding inductively through the generators. 
The second step is an option to reduce the elements of $R_\ell(U)$ further. 
As elements of $Z$ are evaluated
to $0$, this can be applied to all coefficients without changing the word.
Together these reductions produce a reduced form $b_1 l_1 + \cdots + b_n l_n$ 
with coefficients $b_1, \ldots, b_n \in R_\ell(U)$. We denote this reduced 
form by $RF(a_1 l_1 + \cdots + a_n l_n) = b_1 l_1 + \cdots + b_n l_n$.

\begin{theorem}\label{thm6}
Let $\mathcal{P}(A)$ be a symbolic Lie $p$-presentation. If
\begin{items}
\item[\rm (a)]
$RF(J(i,j,k)) = 0$ for all $1 \leq i, j, k \leq n$, and 
\item[\rm (b)]
$RF(P(i,j)) = 0$ for all $1 \leq i,j \leq n$, 
\end{items}
then $\P(A)$ is consistent.
\end{theorem}

\begin{proof}
If the reduced forms evaluate to zero, then every evaluation of them will
evaluate to zero, and hence the result follows from Theorem \ref{thm5}.
\end{proof}

\subsection{Examples}

In the following descriptions we often omit relations of a (symbolic) 
Lie $p$-presentation that have a trivial right-hand side.

\begin{example}
Consider the Lie $p$-presentation of $L$ having the following non-trivial 
relations: 
\[ pl_1=l_2,\qquad pl_2=l_3 \qquad [l_1,l_2]=l_3.  \]
Since $[l_1,l_3]=[l_2,l_3]=0$,
all Jacobi expressions vanish. The only non-trivial compatibility
relation is
$P(1,1)
=p[l_1,l_1]-[pl_1,l_1]
=-[l_2,l_1]
=l_3$.
Hence, $RF(P(1,1))=l_3$, and therefore consistency forces $l_3=0$. With this relation imposed, the
presentation reduces to two generators $\{l_1, l_2\}$, yielding a
cyclic Lie $p$-ring of order $p^2$.
\end{example}

\begin{example}
Consider the Lie $p$-presentation $\mathcal{P}(A)$ 
having the following non-trivial relations:
\[
\begin{aligned}
[l_1,l_2]&=-l_3,\quad [l_1,l_3]=l_5,\quad [l_2,l_3]=l_6,\\
[l_1,l_6]&=l_7,\quad pl_1=l_4,\quad pl_2=l_3,\quad pl_4=l_5-l_6.
\end{aligned}
\] The relevant consistency relations are
\[
J(1,2,3)=l_7,\qquad
P(2,1)=-l_5,\qquad
P(2,2)=-l_6.
\] Therefore \[
RF(J(1,2,3))=l_7,\qquad
RF(P(2,1))=-l_5,\qquad
RF(P(2,2))=-l_6.
\] Thus, consistency forces $l_5=l_6=l_7=0$. Consequently,
$pl_4=l_5-l_6=0$.  This yields
the presentation 
\[
[l_2,l_1]=l_3,\qquad
pl_1=l_4,\qquad
pl_2=l_3
\]
which is consistent and defines a Lie $p$-ring of size $p^4$.
\end{example}

\begin{example}
Let $\mathcal{P}(A)$ be the symbolic Lie $p$-presentation with 
$U=\langle x_1\rangle$.
The non-trivial defining relations are \[
\begin{aligned}
pl_1&=l_4,\quad pl_2=x_1l_5+l_6,\quad [l_2,l_1]=l_3,\quad [l_3,l_1]=l_4,\\
[l_3,l_2]&=l_5,\quad [l_4,l_2]=l_6+x_1l_7,\quad [l_5,l_1]=l_6.
\end{aligned}
\] The relevant consistency relations are
\[
J(1,2,3)=-x_1l_7,\qquad
P(1,2)=-l_6-x_1l_7.
\]
Hence, $RF(J(1,2,3))=-x_1l_7$. Since $x_1\in U$, it is invertible in $R_\ell(U)$. Thus consistency
implies $l_7=0$. Substituting this into $P(1,2)$, we obtain
$RF(P(1,2))=-l_6$,
and hence $l_6=0$. After eliminating $l_6$ and $l_7$,
we obtain the following consistent symbolic Lie \(p\)-presentation
\[
pl_1=l_4,\qquad
pl_2=x_1l_5,\qquad
[l_2,l_1]=l_3,\qquad
[l_3,l_1]=l_4,\qquad
[l_3,l_2]=l_5.
\]
This consistent symbolic Lie \(p\)-presentation defines a family of Lie
\(P\)-rings of order \(P^5\). In general, distinct parameter values may
yield isomorphic Lie \(P\)-rings. For instance, when \(P=5\), the Lie
\(P\)-rings obtained by evaluating \(x_1\) at \(2\) and \(3\) are
isomorphic.
\end{example}

\begin{example}
Let $U=\langle x_1,x_1-2\rangle$ and consider the
symbolic Lie $p$-presentation \[
\begin{aligned}
pl_2&=x_1l_5+l_6,\quad pl_3=l_4+l_5,\quad [l_2,l_1]=l_4,\quad [l_3,l_1]=l_5,\\
[l_5,l_2]&=(x_1-2)l_6,\quad [l_4,l_3]=x_1l_7.
\end{aligned} 
\] The relevant consistency relations are
\[
P(3,2)=-(x_1-2)l_6,\qquad
J(1,2,3)=-(x_1-2)l_6+x_1l_7.
\]
Hence $RF(P(3,2))=-(x_1-2)l_6.$
Since $x_1-2\in U$, it is invertible in $R_\ell(U)$, and therefore
$l_6=0$. Substituting this into the Jacobi relation gives $RF(J(1,2,3))=x_1l_7$.
As $x_1\in U$, it is also invertible in $R_\ell(U)$, and hence
$l_7=0$.
Eliminating the generators $l_6$ and $l_7$ yields the consistent symbolic
Lie $p$-presentation on five generators 
\[
pl_2=x_1l_5,\qquad
pl_3=l_4+l_5,\qquad
[l_2,l_1]=l_4,\qquad
[l_3,l_1]=l_5.
\]
\end{example}

\section{Computing Hermite normal forms}\label{hnf} 

Let $p$ be an indeterminate, let $\ell \in \N$, and let $w, x_1, \ldots, 
x_m$ be commuting indeterminates. As before, we write $R_\ell = 
\Z[1/\ell][w, x_1, \ldots, x_m]$. Let $U\subseteq R_\ell\setminus \{0\}$ 
be a multiplicatively closed subset and $Z$ an ideal of $R_\ell$. Write 
$R_\ell(U)$ for the localization of $R_\ell$ at $U$.

Recall from \ref{symb} that an element $f/g \in R_\ell(U)$ can be evaluated 
at a prime $P$ with $\gcd(P, \ell) = 1$ and $X \in \Sigma_P(U,Z)$ via 
$f(W,X)/g(W,X) \bmod P$. Similarly, a vector or a matrix of elements in 
$R_\ell(U)$ can be evaluated by evaluating every entry.

We consider a free $R_\ell(U)$-module $V = R_\ell(U)^d$ and assume that 
we are given a collection of elements $v_1, \ldots, v_k$ in $V$. Let $M$ 
be the matrix with rows $v_1, \ldots, v_k$. Our aim is to determine an 
integer $\ell' \geq \ell$ and a finite collections of tuples 
$\{ (M_i, U_i, Z_i) \mid i \in I \}$ with $U \subseteq U_i$, 
$Z \subseteq Z_i$, and for each prime $P$ with $\gcd(P, \ell') = 1$
and each $X \in \Sigma_P(U_i, Z_i)$ it follows that the evaluated matrix 
$M_i(X)$ is a basis for the row space of $M(X)$.

We use a variation of the Gaussian elimination on $M$. Our aim is to 
determine a reduced upper echelonised form for $M$ with leading 
coefficients $1$ in each non-zero row. More precisely, there exist 
$c_1 < c_2 < \cdots < c_s$ so that $m_{i,c_i} = 1$ and $m_{j, c_i} = 0$ 
for $j < i$ and $m_{i,k} = 0$ for $k < c_i$. Our algorithm proceeds by 
induction along the columns of $M$. We describe the procedure for the 
first column:

\begin{items}
\item[(1)]
If all entries in the first column in $M$ are zero, then return.
\item[(2)]
For all entries $f/g$ in $M$ reduce $f$ additively modulo $Z$.
\item[(3)]
If there is a non-zero entry in the first column, then permute rows
of $M$ so that this entry is at position (1,1). 
\item[(4)]
If $a_{1,1} = a/b \in \Q$ with $\gcd(a,b)=1$, then 
\begin{items}
\item[(a)]
Modify $\ell' = \operatorname{lcm}(\ell, a)$. 
\item[(b)]
Divide the first row by $a_{1,1}$.
\item[(c)]
Clear out $a_{j,1}$ for $j > 1$ with row operations.
\item[(d)]
return. 
\end{items}
\item[(5)]
If $a_{1,1} = f/g \in R_\ell(U)$, then split into cases:
\begin{items}
\item[Case (I)]
Modify $(U,Z)$ to $(U \cup \{f\},Z)$, divide the first row by $f/g$
and goto (4c).
\item[Case (II)]
Modify $(U,Z)$ to $(U, Z \cup \{f\})$, set $a_{1,1} = 0$ and goto (2).
\end{items}
\end{items}

For practicality, it is an advantage to arrange Step (3) so that 
non-zero integral elements are preferred. Of course, this is not 
always possible.

After completing the algorithm for the first column, the next columns
are processed similarly as in the standard Gaussian elimination algorithm.
The following theorem is immediate. 

\begin{theorem}
Let $(M, U, Z)$ and $\ell$ be the input of the algorithm, and let 
$\{ (M_i, U_i, Z_i) \mid i \in I\}$ and $\ell'$ be the output.
\begin{items}
\item[\rm (1)]
The algorithm terminates and returns with a finite number of cases;
that is, the index set $I$ is finite.
\item[\rm (2)]
In each output triple $(M_i,U_i,Z_i)$, the matrix $M_i$ evaluates to a 
reduced echelon form for each prime $P$ with $\gcd(P, \ell') = 1$ and 
each $X \in \Sigma_P(U_i,Z_i)$.
\item[\rm (3)]
If $X \in \Sigma_P(U,Z)$ for some prime $P$ with $\gcd(P, \ell') = 1$, 
then there exists $i \in I$ with $X \in \Sigma_P(U_i, Z_i)$.
\end{items}
\end{theorem}

\begin{proof}
(1) The first column is cleared after finitely many steps with finitely
many cases. By induction on the number of columns, the result follows. \\
(2) The matrix $M_i$ is in reduced echelon form. Hence its evaluations are
also in reduced echelon form for all possible $P$ and $X$. \\
(3) Each splitting step considers a single $f$ and adds it to $U$ on the
one side (and then allows to divide by $f$) and adds it to $Z$ on the 
other side (and then considers it as zero). As $f(W,X) \bmod P$ is either
non-zero or zero, this yields the result.
\end{proof}

We illustrate the algorithm with the following example.

\begin{example}
Let $\ell = 1$ and $m = 1$, so that $R = \Z[w,x_1]$. Let $U = 
\langle w \rangle$ and $Z = \{0\}$. Consider the matrix $M=\begin{pmatrix}x_1&0&0\\
   0&2&1 
\end{pmatrix}$, whose rows are $v_1 = (x_1,0,0)$ and $v_2 = (0,2,1)$.
 The algorithm branches into two
case distinctions based on the parameter $x_1$:
\begin{items}
\item[(1)] 
Case $x_1 = 0$. In this case, we modify $\ell'= 2$ and
$Z_1 = \langle Z, x_1 \rangle$. We then obtain $M_1=\begin{pmatrix}
    0& 1& 1/2
\end{pmatrix}$ with $U_1=U$ and $Z_1=\langle Z, x_1 \rangle$. For every prime $P\neq 2$ and every $X \in \Sigma_P(U_1, Z_1)$, the row of $M_1(X)$ forms a basis for the row space of $M(X)$. 

\item[(2)] 
Case $x_1 \neq 0$. In this case, we modify $\ell'= 2$
and $U_2 = \langle U, x_1 \rangle$. Since $x_1$ is now invertible, the resulting matrix is $M_2=\begin{pmatrix}1&0&0\\
   0&1&1/2
\end{pmatrix}$ with $U_2=\langle U, x_1 \rangle$ and $Z_2=Z$. For every prime $P\neq 2$ and every $X \in \Sigma_P(U_2, Z_2)$, the rows of $M_2(X)$ forms a basis for the row space of $M(X)$. 
\end{items}
This example demonstrates that the case distinctions are necessary
to get a full overview of the possible results.
\end{example}

\section{The Lie $p$-cover}

Let $L$ be a Lie $p$-ring, and let $d = d(L)$ be its minimal generator 
number. Let $\Lambda(L)$ denote the smallest non-trivial ideal in the
lower $p$-central series of $L$.

Then there exists a presentation $F/R \cong L$, where $F$ is free Lie ring on
$d$ generators, and $R$ is the ideal of defining relations of $L$. The quotient $L^* = F / ([R,F] + pR)$ is called the Lie $p$-cover of 
$L$. As in the group-theoretic setting, the isomorphism 
type of $L^*$ is independent of the chosen presentation $F/R$. 

Since $L$ is finite, the Lie $p$-cover $L^*$ is also finite, its ideal 
$M(L) = R/([R,F] + pR)$ is called the multiplicator of $L$. If $L$ has 
$p$-class $c$, then $N(L) = \Lambda_{c+1}(L^*) \leq M(L)$; the ideal
$N(L)$ is said to be the nucleus of $L$. 
A Lie $p$-ring $H$ is called a descendant of $L$ if $H / \Lambda(H) \cong L$ holds.

The following theorem, which is the Lie-theoretic analogue of the corresponding result for finite $p$-groups, describes the role of the Lie $p$-cover $L^*$ of Lie $p$-ring $L$ in constructing descendants.
\begin{theorem}
Let $L$ be a Lie $p$-ring with Lie $p$-cover $L^*$, multiplicator 
$M(L)$, and nucleus $N(L)$.
\begin{items}
\item[\rm (a)]
If $H$ is a descendant of $L$, then there exists $S \leq M(L)$ 
supplementing $N(L)$ in $M(L)$ such that $L^* / S \cong H$. 
\item[\rm (b)]
If $S$ is a supplement to $N(L)$ in $M(L)$, then $L^* / S$ is a
descendant of $L$.
\item[\rm (c)]
The automorphism group $\mathrm{Aut}(L)$ acts as subgroup of $\mathrm{Aut}(M(L))$ on $M(L)$. 
\item[\rm (d)]
Let $S_1, S_2$ be two supplements to $N(L)$ in $M(L)$. Then $L^*/S_1
\cong L^*/S_2$ if and only if there exists $\alpha \in \mathrm{Aut}(L)$ with
$S_{1}^\alpha = S_2$. 
\end{items}
\end{theorem}

\subsection{Computing $L^*$}
\label{ordcov}

We first consider the case that a Lie $p$-ring $L$ is given by a consistent 
Lie $p$-presentation $\P(A)$ on generators $l_1, \ldots, l_n$. Recall that 
$\mathcal{P}(A)$ has defining relations
\begin{eqnarray*}
 [l_i, l_j] &=& \sum_{k=j+1}^n a_{i,j,k} l_k 
                \mbox{ for } 1 \leq i < j \leq n, \\
  pl_i      &=& \sum_{k=i+1}^n a_{i,i,k} l_k 
                \mbox{ for } 1 \leq i \leq n.
\end{eqnarray*}
We assume that the presentation is {\em weighted} and has {\em definitions}. 
We recall the meaning of these notations as follows.
\bigskip

{\bf Weighted:} A Lie $p$-presentation is called {\em weighted} if each 
generator $l_i$ is assigned a positive integer $w_i$, called its {\em weight},
such that $l_i \in \Lambda_{w_i}(L)\setminus \Lambda_{w_i+1}(L)$. Moreover, 
the series $L=U_1>U_2>\cdots$ obtained by $U_i=\langle l_i,\ldots,l_n\rangle$
refines the lower $p$-central series $(\Lambda_i(L))_{i\geq 1}$ of $L$.
\medskip

{\bf Definitions:} The presentation of $L$ can be structured such that for 
every generator $l_i$ of weight $w_i \geq 2$, there a relation whose 
right-hand side is of the form $u l_i$, where $u \in \Z_p^\times$ is a 
unit. This can be a commutator relation; in this case the relation reads 
$[l_j, l_k] = u l_i$ (where $j<k<i$) or a multiplication relation so that 
$p l_j = u l_i$ (where $j<i$). One of these relations is taken as 
the {\em definition} of $l_i$.  (In most cases, the unit  $u$ is $1$.)
\medskip

A consistent Lie $p$-presentation for the $p$-cover $L^*$ is obtained by extending
the presentation of $L$. As an initial step, we introduce a presentation
$\widehat{\mathcal{P}}(A)$ over generating set $l \cup t$, where $l = 
\{l_1, \ldots, l_n\}$ and $t = \{t_{i,j} \mid 1 \leq i \leq j \leq n\}$ 
represents a collection of new generator, called {\em tails}. The relations 
for $\widehat{\mathcal{P}}(A)$  are given by
\begin{eqnarray*}
 [l_i, l_j] &=& \sum_{k=j+1}^n a_{i,j,k} l_k + t_{i,j}
                \mbox{ for } 1 \leq i < j \leq n, \\
  pl_i      &=& \sum_{k=i+1}^n a_{i,i,k} l_k + t_{i,i}
                \mbox{ for } 1 \leq i \leq n, \\
 {[x, t_{i,j}]} &=& 0 \mbox{ for } 1 \leq i\leq j \leq n
                     \mbox{ and } x \in l \cup t, \\
  p t_{i,j} &=& 0 \mbox{ for } 1 \leq i \leq j \leq n.
\end{eqnarray*}

We note that
\begin{items}
\item[$\bullet$]
If $[l_i, l_j]$ is a definition, then we set $t_{i,j} = 0$.
(In this case $t_{i,j}$ is redundant, since the relation $[l_i,l_j]
= u l_k + t_{i,j}$ can be modified to $[l_i, l_j] = u l_k'$ by 
substituting $l_k' = l_k + \frac{1}{u} t_{i,j}$.)
\item[$\bullet$] 
If $p l_i$ is a definition, then we set $t_{i,i} = 0$ for the same reason.
\item[$\bullet$]
If $L$ has $p$-class $c$, then $L^*$ has $p$-class at most $c+1$. Hence, 
we only need generators of weight at most $c+1$, and we set $t_{i,j} = 0$ 
whenever $w_i + w_j > c+1$.
\end{items}

\begin{theorem}
\label{cover}
$\widehat{\mathcal{P}}(A)$ is a Lie $p$-presentation for $L^*$.
\end{theorem}

\begin{proof}
Let $L$ be a Lie $p$-ring with consistent presentation
$\mathcal{P}(A)$, and let $\hat{L}$ be the Lie $p$-ring defined by the 
presentation $\widehat{\mathcal{P}}(A)$. We show that $\hat{L} \cong L^*$.
Define a map $\pi : \hat{L} \to L$
by $\pi(l_i) = l_i$ for $1 \leq i \leq n$ and $\pi(t_{i,j}) = 0$ for 
all $i,j$.  Then $\pi$ respects the defining relations of 
$\widehat{\mathcal{P}}(A)$, since these reduce to the defining relations 
of $\mathcal{P}(A)$ in $L$. Hence $\pi$ is a well-defined surjective 
homomorphism. 

Let $T = \ker(\pi)$. Then $T$ is generated by the elements $t_{i,j}$. 
From the defining relations, we have
$[x, t_{i,j}] = 0 \;\text{for all generators } x \in l \cup t$,
hence $[\hat{L}, T]=0$, so $T$ is a central ideal of $\hat{L}$. Moreover,
$pT = 0$. Thus, we obtain an extension
$0 \longrightarrow T \longrightarrow \hat{L} \xrightarrow{\pi} L 
   \longrightarrow 0$
with $[\hat{L},T]=0$ and $pT=0$.
\medskip

\noindent
Let $0 \to N \to E \xrightarrow{\psi} L \to 0$ be any central extension 
with $pN=0$.  Choose  a preimage $\tilde{l}_i \in E$ of $l_i$. Since $\psi$ 
is a homomorphism, the image of the relations of $L$ in $E$ must lie in 
the kernel $N$. Specifically 
\[
[\tilde{l}_i,\tilde{l}_j] = \sum_{k=j+1}^n a_{i,j,k} \tilde{l}_k + n_{i,j},
\quad n_{i,j} \in N,
\]
\[
p\tilde{l}_i = \sum_{k=i+1}^n a_{i,i,k} \tilde{l}_k + n_{i,i},
\quad n_{i,i} \in N.
\]
Define a homomorphism $\varphi : \hat{L} \to E$ by
\[
\varphi(l_i) = \tilde{l}_i, \quad \varphi(t_{i,j}) = n_{i,j}.
\]
Since $N$ is central and $pN=0$, all defining relations of 
$\widehat{\mathcal{P}}(A)$ are satisfied in $E$, and hence $\varphi$ 
is a well-defined homomorphism, and by construction $\psi \circ \varphi = \pi$. A central extension $\hat{L}$ satisfies the universal property of the 
$p$-cover $L^*$, so $\hat{L} \cong L^*$. Therefore, $\widehat{\mathcal{P}}(A)$ 
is a Lie $p$-presentation for the $p$-cover $L^*$.
\end{proof}

Theorem \ref{cover} effectively allows to determine a Lie $p$-presentation
for $L^*$ from a consistent Lie $p$-presentation of $L$. It remains to 
make the resulting presentation for $L^*$ consistent. To achieve this, we use Theorem \ref{thm5} and evaluate the power reduced forms described there. 
Since the presentation of $L$ is consistent, each power reduced form 
determines an element of $T = \langle t_{i,j} \mid 1 \leq i \leq j \leq n
\rangle$. By construction, $T$ is a vector space with basis $t=\{ t_{i,j}
\mid 1 \leq i \leq j \leq n \}$, and expressing each power reduced form with respect to basis $t$ yields a collection of coordinate vectors 
$v_1, \ldots, v_k$ in $T$. Factoring out the subspace generated by these vectors yields the quotient presentation
\[ \widehat{\P}(A) / \langle v_1, \ldots, v_k \rangle,\]
which is the desired consistent Lie $p$-presentation for $L^*$.

\subsection{Computing Lie $p$-covers in the symbolic case}
\label{symcov}

We now consider the symbolic setting. Consider a Lie $p$-ring $L$ given by the symbolic Lie $p$-presentation $\P(A)$. We assume that the 
conditions (a) and (b) of Theorem \ref{thm6} hold in $\P(A)$; note that
this is the case for each symbolic Lie $p$-presentation in the 
library of Lie $p$-rings of dimension at most $7$ in the Lie 
$p$-ring package.

The construction described in Section \ref{ordcov} extends naturally to this symbolic setting. We first construct a symbolic Lie $p$-presentation $\widehat{\P}(A)$ for 
$L^*$ by extension of $\P(A)$. Then we evaluate conditions (a) and 
(b) of Theorem \ref{thm6}. Since these relations vanish in $L$, each consistency condition determines an element of the multiplicator $M(L)$. Hence we obtain a collection of elements $v_1, \ldots, v_k$ in $M(L)$, where each $v_i$ is expressed as a linear combination of the tail generators. Writing these elements with respect to the chosen basis of $M(L)$ gives the matrix $M$.

We now apply the Hermite normal form algorithm described in Section \ref{hnf},
and thus determine a finite set $\{(M_i, U_i, Z_i) \mid i \in I\}$
describing the different row reduced echelon forms arising for 
$\langle v_1, \ldots, v_k \rangle$. For each one of the matrix $M_i$, the
row reduced echelon form allows to compute the quotient $L^*/\langle M_i \rangle$. In this way we obtain a finite list of consistent symbolic Lie $p$-presentations
of the Lie $p$-covers of $L$.

\begin{example}
Let $L$ be a Lie $p$-ring of order $p^5$ specified by the consistent 
symbolic Lie $p$-presentation $\mathcal{P}(A)$ with the non-trivial 
defining relations:
\[
pl_1=l_4,\quad
pl_2=xl_5,\quad
[l_1,l_2]=-l_3,\quad
[l_1,l_3]=-l_4,\quad
[l_2,l_3]=-l_5.
\]
In this case, $R_\ell=\Z[1/\ell][w,x]$, $U=\langle x\rangle$, and $Z=\{0\}$. Following the construction of Section~\ref{ordcov}, we obtain the symbolic Lie $p$-presentation $\widehat{\mathcal{P}}(A)$ by adjoining central tail generators $t = \{t_{i,j} \mid 1 \leq i \leq j \leq 5\}$ (set $t_{1,2}= t_{1,3}= t_{2,3}=0$) satisfying $[k, t_{i,j}] = 0$ for all $k \in \{l_1,\ldots, l_5\} \cup t$ and $pt_{i,j} = 0$: 
\begin{align*}
pl_1 &= l_4+t_{1,1} & [l_1, l_2] &= -l_3 & [l_2, l_4] &= t_{2,4} \\
pl_2 &= xl_5+t_{2,2} & [l_1, l_3] &= -l_4 & [l_2, l_5] &= t_{2,5} \\
pl_3 &= t_{3,3} & [l_1, l_4] &= t_{1,4} & [l_3, l_4] &= t_{3,4} \\
p l_4 &= t_{4,4} & [l_1, l_5] &= t_{1,5} & [l_3, l_5] &= t_{3,5} \\
p l_5 &= t_{5,5} & [l_2, l_3] &= -l_5  &[l_4, l_5] &= t_{4,5} &
\end{align*}
The consistency conditions of $\widehat{\mathcal{P}}(A)$ are obtained from 
the relations $RF(P(i,j)) = 0$ and the Jacobi identities 
$RF(J(i, j, k)) = 0$ for all $1 \leq i, j, k \leq 5$. These determine the 
following elements $v_1,\ldots,v_{13}$ in $M(L)$:
\[
\begin{aligned}
v_1&=P(1,1)= t_{14}, & v_2&=P(1,2)=t_{2,4}-t_{3,3}\\
v_3&=P(1,3)=t_{3,4}-t_{4,4},&v_4&=P(1,5)=J(5,3,1)=-J(4,3,2)=-t_{4,5}\\
v_5&=P(2,1)=t_{3,3}+xt_{1,5}, & v_6&=P(2,2)=xt_{2,5}\\
v_7&=P(2,3)= xt_{3,5}-t_{5,5}, &
v_8&=P(2,4)= xt_{4,5}\\
v_9&=P(3,1)=t_{4,4},&
v_{10}&=P(3,2)=t_{5,5}\\
v_{11}&=J(3,2,1)=t_{1,5}-t_{2,4},&
v_{12}&=J(4,2,1)=-t_{3,4}\\
v_{13}&=J(5,2,1)=-t_{3,5}.
\end{aligned}
\]
With respect to the ordered basis $\{t_{3,3}, t_{4,4}, t_{5,5}, t_{1,5}, t_{2,4}, t_{3,4}, t_{3,5}, t_{4,5}, t_{1,4}, t_{2,5}\}$, we obtain the following matrix $M$: 
\[
M = \left( \begin{array}{cccccccccc}
t_{3,3} & t_{4,4} & t_{5,5} & t_{1,5} & t_{2,4} & t_{3,4} & t_{3,5} & t_{4,5}&t_{1,4}&t_{2,5} \\ \hline
-1 & 0 & 0 & 0 & 1 & 0 & 0 & 0 &0&0 \\
0 & -1 & 0 & 0 & 0 & 1 & 0 & 0 &0&0 \\
0 & 0 & 1 & 0 & 0 & 0 & 0 &0&0 &0\\
0 & 0 & 0 & 1 & -1 & 0 & 0 & 0 &0&0\\
0 & 0 & 0 & 0& 0 & -1 & 0 & 0 &0&0\\
0 & 0 & 0 & 0& 0 & 0 & -1 & 0 &0&0\\
0 & 0 & 0 & 0& 0 & 0 & 0 & -1 &0&0\\
0 & 0 & 0 & 0& 0 & 0 & 0 & 0 &1&0\\
1 & 0 & 0 & x & 0 & 0 & 0 & 0&0&0 \\
0 & 0 & 0 & 0 & 0 & 0 & 0 & 0&0&x \\
0 & 0 & 0 & 0 & 0 & 0 & 0 & x&0&0 \\
0 & 1 & 0 & 0 & 0 & 0 & 0 & 0 &0&0\\
0 & 0 & -1 & 0 & 0 &0& x & 0  &0&0
\end{array} \right)
\]
Applying the row-reduction algorithm of Section~\ref{hnf} to $M$ produces the following matrix:
\[ \left( \begin{array}{cccccccccc}
t_{3,3} & t_{4,4} & t_{5,5} & t_{1,5} & t_{2,4} & t_{3,4} & t_{3,5} & t_{4,5}&t_{1,4}&t_{2,5} \\ \hline
1 & 0 & 0 & 0 & -1 & 0 & 0 & 0 &0&0 \\
0 & 1 & 0 & 0 & 0 & 0 & 0 & 0 &0&0 \\
0 & 0 & 1 & 0 & 0 & 0 & 0 &0&0 &0\\
0 & 0 & 0 & 1 & -1 & 0 & 0 & 0 &0&0\\
0 & 0 & 0 & 0 & 1+x & 0 & 0 & 0&0&0 \\
0 & 0 & 0 & 0& 0 & 1 & 0 & 0 &0&0\\
0 & 0 & 0 & 0& 0 & 0 & 1 & 0 &0&0\\
0 & 0 & 0 & 0& 0 & 0 & 0 & 1 &0&0\\
0 & 0 & 0 & 0& 0 & 0 & 0 & 0 &1&0\\
0 & 0 & 0 & 0 & 0 & 0 & 0 & 0&0&x \\
0 & 0 & 0 & 0 & 0 & 0 & 0 & 0&0&0 \\
0 & 0 & 0 & 0 & 0 & 0 & 0 & 0 &0&0\\
0 & 0 & 0 & 0 & 0 &0& 0 & 0  &0&0
\end{array} \right)
\]

\textbf{Case 1:} In this case, $\ell_1=\ell, \;Z_1=\{0\}$, and modify $U_1=\langle x,x+1\rangle$. Since $x$ and $x+1$ are invertible in $R_{\ell_1}(U_1)$, the algorithm yields $M_1=\begin{pmatrix}
    I_{10}
\end{pmatrix}$, where $I_{10}$ denotes the $10\times 10$ identity matrix. Hence all tails occurring in the matrix $M_1$ vanish. Thus, the basis for the quotient $L^{*} / \langle M_1\rangle$ is $\{\overline{t}_{1,1}, \overline{t}_{2,2}\}$, where  
$\overline{t}_{i,i}$ denotes the image of $t_{i,i}$ in the quotient. Writing
$l_6=\overline{t}_{1,1},\;
l_7=\overline{t}_{2,2}$, the resulting consistent symbolic Lie $p$-presentation $\widehat{\mathcal{P}}(A)$ is \[
pl_1=l_4+l_6,\;
pl_2=-l_5+l_7,\;
[l_1,l_2]=-l_3,\;
[l_1,l_3]=-l_4,\;
[l_2,l_3]=-l_5.
\]

\textbf{Case 2:} In this case, $\ell_2=\ell$, $U_2=\langle x \rangle$, and modify $Z_2=\langle Z, x+1\rangle=\langle x+1\rangle$. Reducing the matrix modulo $Z_2$, we get \[M_2= \left( \begin{array}{cccccccccc}
t_{3,3} & t_{4,4} & t_{5,5} & t_{1,5} & t_{2,4} & t_{3,4} & t_{3,5} & t_{4,5}&t_{1,4}&t_{2,5} \\ \hline
1 & 0 & 0 & 0 & -1 & 0 & 0 & 0 &0&0 \\
0 & 1 & 0 & 0 & 0 & 0 & 0 & 0 &0&0 \\
0 & 0 & 1 & 0 & 0 & 0 & 0 &0&0 &0\\
0 & 0 & 0 & 1 & -1 & 0 & 0 & 0 &0&0\\
0 & 0 & 0 & 0 & 0 & 0 & 0 & 0&0&0 \\
0 & 0 & 0 & 0& 0 & 1 & 0 & 0 &0&0\\
0 & 0 & 0 & 0& 0 & 0 & 1 & 0 &0&0\\
0 & 0 & 0 & 0& 0 & 0 & 0 & 1 &0&0\\
0 & 0 & 0 & 0& 0 & 0 & 0 & 0 &1&0\\
0 & 0 & 0 & 0 & 0 & 0 & 0 & 0&0&-1 \\
0 & 0 & 0 & 0 & 0 & 0 & 0 & 0&0&0 \\
0 & 0 & 0 & 0 & 0 & 0 & 0 & 0 &0&0\\
0 & 0 & 0 & 0 & 0 &0& 0 & 0  &0&0
\end{array} \right)
\]
Here, $t_{3,3}=t_{1,5}=t_{2,4}$, while all other tails appearing in $M_2$ matrix vanish. Thus, a basis for the quotient $L^{*} / \langle M_2\rangle$ is
$\{\overline{t}_{1,1}, \overline{t}_{2,2}, \overline{t}_{2,4}\}$.
Equivalently, one may choose $\overline{t}_{3,3}$ or $\overline{t}_{1,5}$ as the basis element, since
$\overline{t}_{3,3}
=
\overline{t}_{1,5}
=
\overline{t}_{2,4}$
in the quotient $L^{*} / \langle M_2\rangle$. Setting $l_6=\overline{t}_{1,1},\;
l_7=\overline{t}_{2,2},\;
l_8=\overline{t}_{2,4}$, the presentation $\widehat{\mathcal{P}}(A)$ simplifies to
\[
pl_1=l_4+l_6,\;
pl_2=-l_5+l_7,\;
pl_3=l_8,\]
\[
[l_1,l_2]=-l_3,\;
[l_1,l_3]=-l_4,\;
[l_2,l_3]=-l_5,\;
[l_2,l_4]=l_8,\;
[l_1,l_5]=l_8.
\]
\end{example}

\subsection*{Acknowledgment} Renu Joshi acknowledges support from the Alexander von Humboldt Foundation through the Humboldt Research Fellowship for Postdocs (Humboldt ID number: 1243434).

\smallskip

\noindent \textbf{Statements and Declarations.}

\smallskip

\noindent  {\bf Data availability} Data sharing does not apply to this article as no datasets were analyzed during its preparation.\\
\noindent   {\bf Conflict of interests} The authors affirm that they have no financial interests to disclose.

\providecommand{\bysame}{\leavevmode\hbox to3em{\hrulefill}\thinspace}
\providecommand{\MR}{\relax\ifhmode\unskip\space\fi MR }
\providecommand{\MRhref}[2]{%
  \href{http://www.ams.org/mathscinet-getitem?mr=#1}{#2}
}
\providecommand{\href}[2]{#2}


\begin{thebibliography}{1}

\bibitem{Gr00}
Willem~A. de~Graaf, \emph{Lie algebras: theory and algorithms}, North-Holland
  Mathematical Library, vol.~56, North-Holland Publishing Co., Amsterdam, 2000.

\bibitem{VLE15}
B.~Eick and M.~Vaughan-Lee, \emph{{LiePRing} -- {D}atabase and {A}lgorithms for
  {L}ie p-rings}, 2015, GAP 4 package.

\bibitem{GAP}
The GAP~Group, \emph{{GAP -- Groups, Algorithms, and Programming, Version
  4.15.1}}, 2025.

\bibitem{Hig60}
Graham Higman, \emph{Enumerating {$p$}-groups. {II}. {P}roblems whose solution
  is {PORC}}, Proc. London Math. Soc. (3) \textbf{10} (1960), 566--582.

\bibitem{Hup67}
B.~Huppert, \emph{Endliche {G}ruppen. {I}}, Die Grundlehren der mathematischen
  Wissen\-schaf\-ten, vol. 134, Springer-Verlag, Berlin-New York, 1967.

\bibitem{OBr90}
E.~A. O'Brien, \emph{The {$p$}-group generation algorithm}, J. Symbolic Comput.
  \textbf{9} (1990), no.~5-6, 677--698, Computational group theory, Part 1.

\bibitem{OBVL05}
E.~A. O'Brien and M.~R. Vaughan-Lee, \emph{The groups with order {$p^7$} for
  odd prime {$p$}}, J. Algebra \textbf{292} (2005), no.~1, 243--258.

\end{thebibliography}
\end{document}